\documentclass{article}%
\usepackage{amsmath, amssymb}%
\usepackage{amsfonts}%
\usepackage{amssymb}%
\usepackage{graphicx}
\usepackage{tikz}
\usetikzlibrary{patterns}
\newtheorem{theorem}{Theorem}[section]
\newtheorem{lemma}[theorem]{Lemma}

\newtheorem{definition}[theorem]{Definition}
\newcommand{\qed}{\rule{1mm}{3mm}}     

\newenvironment{proof}{\vspace*{\parsep}\noindent {\bf Proof:}}{\qed\\[1em]}

\begin{document}

\title {Principles Weaker than BD-N}

\author{Robert S. Lubarsky and Hannes Diener \\ Dept. of Mathematical
Sciences \\ Florida Atlantic University \\ Boca Raton, FL 33431 \\
Robert.Lubarsky@alum.mit.edu \\ Department Mathematik, Fak. IV \\
Emmy-Noether-Campus, Walter-Flex-Str. 3 \\ University of Siegen \\
57068 Siegen, Germany \\ diener@math.uni-siegen.de} \maketitle
\begin{abstract}
BD-N is a weak principle of constructive analysis. Several
interesting principles implied by BD-N have already been
identified, namely the closure of the anti-Specker spaces under
product, the Riemann Permutation Theorem, and the Cauchyness of
all partially Cauchy sequences. Here these are shown to be
strictly weaker than BD-N, yet not provable in set theory alone
under constructive logic.
\\ {\bf keywords:} anti-Specker spaces, BD-N, Cauchy
sequences, partially Cauchy, Riemann Permutation Theorem,
topological models\\{\bf AMS 2010 MSC:} 03F60, 03F50, 03C90, 26E40
\end{abstract}

\section{Introduction}
BD-N, first identified in \cite {I92} with a pre-history in \cite
{I91}, has turned out to be an important foundational principle,
being equivalent to many statements in analysis having to do with
continuity \cite {I01, IS, IY}. It is also weak, or deep, in that
it holds in all major traditions of mathematics, including
classical mathematics, intuitionism, and Russian constructivism
(based on computability). As such a central and weak principle, it
is reasonable to guess that most particular consequences of it
would either imply it or be outright provable. A moment's sober
reflection would lead one to realize that BD-N is not an atom, or
co-atom, in the Heyting algebra of statements, and that it would
almost certainly just be a matter of time before natural
intermediate statements were found. This note reports on exactly
that.

The first example is the closure of the anti-Specker spaces under
products. (For background, see \cite {BeB1, BeB2, Br09}.) Douglas
Bridges showed such closure under BD-N \cite {Br}, speculating
that they were equivalent. It was shown in \cite {RSL} that in
fact such closure does not imply BD-N. This still leaves open the
question as to whether this closure is provable outright. We show
below it is not.

The next example is the Riemann Permutation Theorem. For
background, see \cite {BeB3, BBDS}, where it is shown, among other
things, that BD-N implies the RPT. We show that RPT does not
itself imply BD-N, in a very similar manner to the anti-Specker
property, and also that it is not itself provable.

Our final example has to do with a weakening of the definition of
a Cauchy sequence, which was identified by Fred Richman (notes),
who called such sequences partially Cauchy. He showed that under
BD-N, all partially Cauchy sequences are Cauchy. A similar kind of
sequence, an almost Cauchy sequence, has also been identified
\cite {BBP}. There it was shown that, under Countable Choice, all
almost Cauchy sequences are Cauchy iff BD-N holds. It is shown
below that the Cauchyness of all partially Cauchy sequences does
not imply BD-N, and also that the Cauchyness of all partially
Cauchy sequences is not itself provable on the basis of set theory
alone.

We speculate that there is an intricate and interesting world of
unprovable principles strictly weaker than BD-N. To investigate
this, several tasks need fulfilling. For one, we would like to see
just how interesting the ones discussed here are, by how many
interesting statements they're equivalent with. We would also like
to see other such intermediate statements. Of course, we need to
know whether these intermediate properties are themselves mutually
inequivalent, and, if so, what implications might hold under
additional hypotheses, such as Countable Choice. As for an
independence result in the general case, the obvious place to look
first would be the models contained in this paper. These
topological models have a claim at being the generic models for
their specific purposes. As such, they naturally tend to keep
principles not intended to be falsified true. So, for instance,
you'd expect that, in the model falsifying RPT, partially Cauchy
sequences would still be Cauchy, unless of course the latter
assertion implied RPT. We do not attempt a thorough analysis of
all of these issues here, preferring to leave this for future
work.

Regarding the methods employed, any independence result of the
form ``A does not imply B" is shown here by providing a model of A
in which B is false. These models are all topological models,
which works essentially like forcing from classical set theory
when you leave out that part of the basic theory where you mod out
by the double negation, the purpose of which is only to model
classical instead of just constructive logic, clearly a move which
is anathema to our purposes. For background on topological models,
see \cite {G1, G2} and the addendum to \cite {RSL10}, or the brief
discussion before theorem 2.3 below.

It would be interesting to see how these issues would play out
with realizability models. The first models discovered falsifying
BD-N were of this kind \cite {Bee, BISV, L}. Each and every one of
them also exhibits a separation of the kind proved here, depending
on which among the anti-Specker closure property, RPT, and the
partially Cauchy property hold in it. The extra challenge
presented by this context is that realizability models seem not to
be the canonical models for these properties, and they're less
flexible to deal with. By way of illustration, most of the
topological models presented here and in \cite {RSL} were not that
difficult to come up with; in contrast, it seems completely
unclear how to concoct a realizability model for the same
purposes. As another illustration, as argued for in \cite {RSL},
topological models seem to be canonical for their purposes, in
part because ground model properties tend to persist into the
topological models, except of course for those that imply the
property purposely being falsified. For instance, the three
properties considered here are all true in the topological model
of not BD-N, as predicted. In contrast, in the realizability
models at hand, all bets are off as to which of those three hold
there. The experts do not have a clear expectation of this
outcome, and furthermore, at least in the one case tried (RPT in
extensional K$_1$ realizability), they have not been able to prove
one way or the other whether it holds. On the other hand, many of
these models are naturally occurring in and of themselves. Hence
it would be nice to know which of the principles under
consideration hold where, in order to understand these models
better, as well as the computational content of the principles
themselves.

The paper is organized as followed. Anti-Specker spaces are
discussed in sec. 2. It was shown already in \cite {RSL} that
their closure under Cartesian product does not imply BD-N; here we
see that such closure can fail under standard set theory (IZF).
The reason we work over IZF is twofold. It is the closest
constructive correlate to ZFC, the de facto gold standard in
mathematics, and it is strictly stronger than the other theories
commonly considered, such as CZF and BISH, so that an independence
result of IZF implies the same over these others. The following
two sections are about the Riemann Permutation Theorem, first that
it does not imply BD-N, and then that it can fail even under IZF.
The two sections after that show the corresponding results for the
assertion that all partially Cauchy sequences are Cauchy.

\section {Anti-Specker spaces may not be closed under products}
An anti-Specker space for our purposes is a metric space $X$ such
that, when you enlarge $X$ by adding a single point $*$ at a
distance of 1 away from every $x \in X$, then every countable
sequence through $X \cup \{*\}$ which is eventually apart from
every point of $X$ is eventually equal to $*$. (Actually, there
are various such anti-Specker properties, sometimes inequivalent,
parametrized by how the space $X$ is extended. Since we consider
here only this one version, we suppress mention of this choice in
the notation and terminology.) Anti-Speckerhood is a form of
compactness. As such, one might reasonably expect anti-Specker
spaces to be closed under Cartesian product. We produce a
topological space $T$ such that the (full) model over $T$
falsifies such closure.

\begin {definition}
Let $T$ consist of $\omega$-sequences $(z_n)$ such that finitely
many entries are pairs of real numbers $\langle x_n, y_n \rangle$
and the rest are $*$, which is taken by convention to equal
$\langle *, * \rangle$ (so every entry has both projections). We
give the topology by describing a sub-basis. An open set in the
sub-basis is given by the following information. The positive
information is a finite sequence $\alpha_n \; (n<N)$, each entry
of which is either $*$ or a pair of finite open intervals $\langle
I_n, J_n \rangle$. A sequence $(z_n)$ satisfies this positive
constraint if $z_n = *$ whenever $\alpha_n = *$ and $z_n \in I_n
\times J_n$ otherwise ($n < N$). The negative information is an
assignment to each of finitely many closed and bounded sets $C_i
(i \in I, I$ an index set) in $\mathbb{R}^2$ of a natural number
$M_i$. This negative information is satisfied by $(z_n)$ if, for
all $n > M_i, \; z_n \not \in C_i$ (where $* \not \in
\mathbb{R}^2$). (Notice that the empty set is given by the
intersection of two sub-basic open sets with incompatible positive
information.)
\end {definition}

An open set is said to be in normal form if the following
conditions hold. For one, for $m, n < N,$ either $\langle I_m, J_m
\rangle = \langle I_n, J_n \rangle$ or $\overline {I_m \times
J_m}$ and $\overline {I_n \times J_n}$ are disjoint (where
$\overline X$ is the closure of $X$). Also, $I$ is a singleton --
that is, the negative information has only one closed set -- and
that unique closed set $C$ is a (necessarily finite) rectangle.
Finally, for $m, n < N \; I_m \times J_n \subseteq C$.
(Implicitly, when reference is made to $\langle I_n, J_n \rangle$
when $\alpha_n = *$, that clause does not apply.)

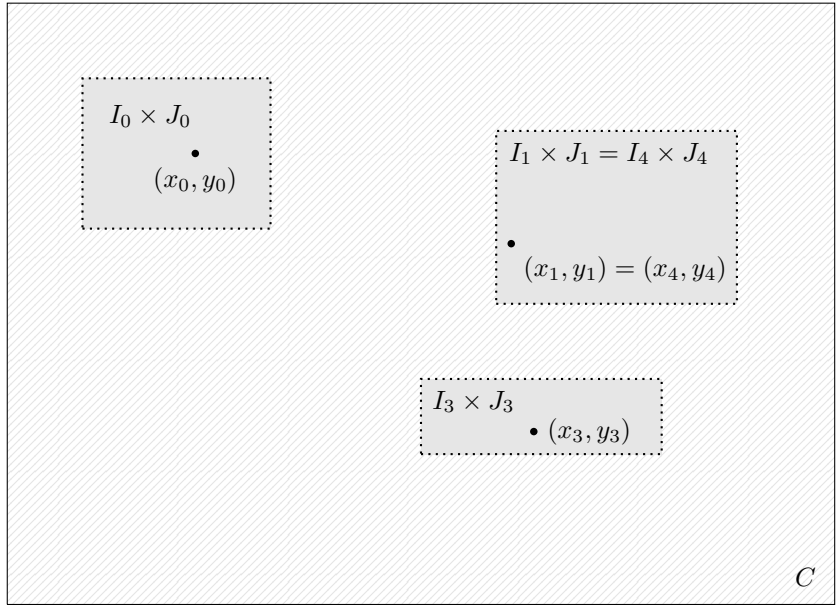
\begin{figure}[h]
\begin{center}
\usetikzlibrary{patterns}
\begin{tikzpicture}
\tikzstyle{rechteck}=[dotted,fill =black!10, thick]
  \draw[pattern=north east lines,pattern color=black!10] (0,0) --  (11,0)  node [label=above left:{$C$}] {} -- (11,8) -- (0,8) -- (0,0) ;
  \draw[rechteck] (1,5) rectangle (3.5,7);
  \draw (1.9,6.5) node {$I_{0} \times J_{0}$};
  \node [circle,fill=black,inner sep=1pt,label={below:$(x_0, y_0)$}] at (2.5,6) {};
  \draw[rechteck] (6.5,4) rectangle (9.7,6.3);
  \draw (8,6) node {$I_{1} \times J_{1} = I_{4} \times J_{4}$};
  \node [circle,fill=black,inner sep=1pt,label={below right:$(x_1, y_1)=(x_{4},y_{4})$}] at (6.7,4.8) {};
  \draw[rechteck] (5.5,2) rectangle (8.7,3);
  \draw (6.2,2.7) node {$I_{3} \times J_{3}$};
  \node [circle,fill=black,inner sep=1pt,label={right:$(x_3, y_3)$}] at (7,2.3) {};

\end{tikzpicture}
\end{center}
\caption{Open set in normal form containing the point $\left(
(x_{0},y_{0}), (x_{1},y_{1}), (x_{2},y_{2}), \dots  \right)$ with
$(x_{2},y_{2}) = \ast$.}
\end{figure}

\begin {lemma} The opens in normal form constitute a sub-basis.
\end {lemma}

\begin {proof}
Given $(z_n) \in O$ extend the positive information to include all
of $(z_n)$'s non-$*$ entries. Then for $\langle x_m, y_m \rangle =
\langle x_n, y_n \rangle$ shrink $I_m \times J_m$ and $I_n \times
J_n$ to be equal; for $\langle x_m, y_m \rangle \not = \langle
x_n, y_n \rangle$ shrink $I_m \times J_m$ and $I_n \times J_n$ to
satisfy the disjointness condition. Furthermore, if $n > M_i$ then
$I_n \times J_n$ must be shrunken to be disjoint from $C_i$. Then
enclose all of the $C_i$'s by one rectangle $C$, also large enough
to cover each $I_m \times J_n$, and assign to $C$ the length of
the positive sequence.
\end {proof}

Let $G$ be the generic. To help make this paper somewhat
self-contained, the basics of topological models include that the
universe of the extension consists exactly of terms, which are
sets of the form $\{ \langle O_i, \sigma_i \rangle \mid i \in I
\}$, where $O_i$ is an open set, $\sigma_i$ inductively a term,
and $I$ an index set. When each $O_i$ hereditarily is the entire
space, then the term is the canonical image $\hat{x}$ of a ground
model set $x$. The generic $G$ is the term $\{ \langle O, \hat{O}
\rangle \mid O$ an open set of $T\}$, which in this case can be
identified with a sequence $(g_n)$. Let $X$ be the set of reals
from the first components of the $g_n$'s, and $Y$ the reals from
the second components.

\begin {theorem} T $\Vdash$ ``X and Y are anti-Specker spaces,
and $X \times Y$ is not." \end {theorem}

\begin {proof}
Clearly, $T \Vdash ``(g_n)$ is a sequence through $X \times Y \cup
\{*\}$." By considering the normal opens, $(g_n)$ is eventually
apart from each point in $X \times Y$. In greater detail, suppose
$(z_n) \in O \Vdash (x,y) \in X \times Y.$ Then some neighborhood
of $(z_n)$ forces $x$ to be in some $I_m$ and $y$ to be in some
$J_n$. Let $U$ be a normal open subset of that neighborhood
containing $(z_n)$. If $U$'s positive information has length $N$,
then $U \Vdash ``$Beyond $N \; (g_n)$ is apart from $(x,y)$."

Also, no open set forces $(g_n)$ eventually to be $*$, because the
closed sets in the negative information are finite. That is, given
any open set $O$ and natural number $k$, there is member of $O$
with a non-$*$ entry beyond slot $k$.

Hence $(g_n)$ witnesses that $X \times Y$ is not an anti-Specker
space.

All that remains to show is that $X$ and $Y$ are anti-Specker
spaces. We will show this for $X$, the case for $Y$ being
symmetric.

To this end, suppose $O \Vdash ``(a_n)$ is a sequence through $X
\cup \{*\}$ eventually apart from each point in $X$." For $(z_n)
\in O$ we must find a neighborhood of $(z_n)$ forcing a place
beyond which $(a_n)$ is always $*$. First extend (i.e. shrink) $O$
so the positive information $\alpha$ contains all of $(z_n)$'s
non-$*$ entries. Then we claim we can extend again to force an
integer $K$ beyond which (i.e. for $k > K$) $a_k$ is apart from
each $x_n$ in $(z_n)$'s non-$*$ entries, all the while keeping
$(z_n)$ in the open set. That is, for each $\alpha_n$ of the form
$\langle I_n, J_n \rangle$, $a_k$ is forced to be at least some
fixed rational distance $r_n$ away from the real approximated by
$I_n$. To do this, iteratively extend the open set to have this
property for each $\langle I_n, J_n \rangle$ individually.

Then extend again by shrinking $I_n$ (to an interval we will still
call $I_n$, recycling notation) so that $I_n$ has length less than
$r_n$. This forces $a_k$ to be apart from the entire interval
$I_n$; even more, $a_k$ is forced not to be in some open interval
containing $I_n$'s endpoints, some extension of $I_n$ both upwards
and downwards. We call such a lengthened interval a forbidden
zone.

Finally, extend yet again to an open set $U \ni (z_n)$ in normal
form, with positive information given by $\alpha$ of length $N$
and negative information given by $C$.

\begin{figure}[t]
\begin{center}
\usetikzlibrary{patterns}
\begin{tikzpicture}
\tikzstyle{rechteck}=[dotted,fill =black!10, thick]
  \draw[pattern=north east lines,pattern color=black!10] (0,0) --  (11,0)  node [label=above left:{$C$}] {} -- (11,8) -- (0,8) -- (0,0) ;
  \draw[rechteck] (1,5) rectangle (3.5,7);
  \draw (1.9,6.5) node {$I_{0} \times J_{0}$};
  \node [circle,fill=black,inner sep=1pt,label={below:$(x_0, y_0)$}] at (2.5,6) {};
  \draw[rechteck] (6.5,4) rectangle (9.7,6.3);
  \draw (8,6) node {$I_{1} \times J_{1} = I_{4} \times J_{4}$};
  \node [circle,fill=black,inner sep=1pt,label={below right:$(x_1, y_1)=(x_{4},y_{4})$}] at (6.7,4.8) {};
  \draw[rechteck] (5.5,2) rectangle (8.7,3);
  \draw (6.2,2.7) node {$I_{3} \times J_{3}$};
  \node [circle,fill=black,inner sep=1pt,label={right:$(x_3, y_3)$}] at (7,2.3) {};

  \node at (2,9) {forbidden zone};
  \node at (6.5,9) {forbidden zone};
  \node [circle,fill=black,inner sep=1pt,label={right:$(x, y)$}] (xy) at (2,8.5) {};
   \node [circle,fill=black,inner sep=1pt,label={below:$v_{\ell}$}] (vl) at (-0.5,5) {};
  \draw (0.5, 10) -- (0.5,-1);
  \draw (4, 10) -- (4,-1);
  \draw (5, 10) -- (5,-1);
  \draw (10, 10) -- (10,-1);
  \pgfsetcornersarced{\pgfpoint{17pt}{17pt}}
  \draw  (xy) --  (-0.5,8.5) -- (vl);
  \pgfsetcornersarced{\pgfpointorigin}

\end{tikzpicture}
\end{center}
\caption{The same open set as before, with forbidden zones and
path P.}
\end{figure}
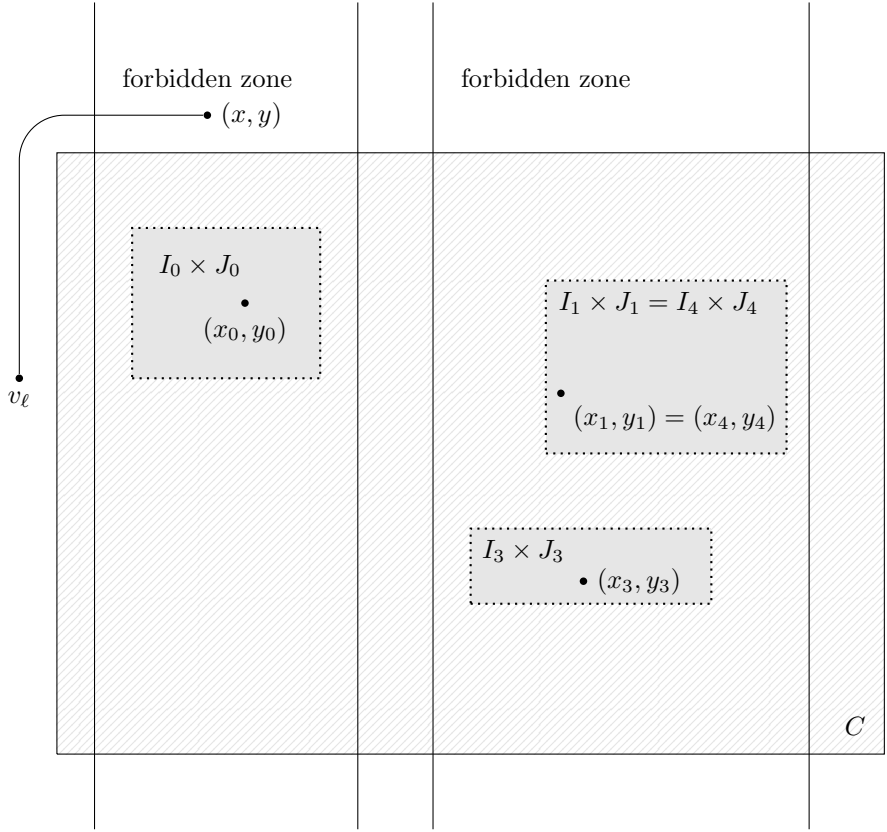

The claim is that $U \Vdash ``$For $k > K \; a_k = *"$. If not,
for some fixed $k > K$ and $l$ let some extension $V$ of $U$ force
$``a_k = x_l"$. Call $V$'s positive information $\beta$. Notice
that, by the construction above, $l > N$. That means we can change
$\beta_l$ without violating $U$'s positive information. Pick some
$v = (v_n) \in V$. Let $P$ be a path in $\mathbb{R}^2$ starting at
$v_l$, ending at some $(x,y)$ with $x$ in some forbidden zone, and
avoiding $C$; this is possible, because $C$ is just a finite
rectangle. For each $p \in P$ let $v(p)$ be identical with $v$
except that $v_l$ is replaced by $p$. Notice that $v(p) \in U$,
because by avoiding $C$ we're also not violating $U$'s negative
information. So around each $v(p)$ is an open subset of $U$
forcing either that $a_k$ is $*$ or that it's not. Let $P_*$ be
$\{ p \mid$ some neighborhood of $v(p)$ forces $``a_k = * \}$ and
$P_x$ be $\{ p \mid$ some neighborhood of $v(p)$ forces $``a_k
\not =* \}$. Since positive information is given by open sets,
whatever some neighborhood of some $v(p)$ forces, the same
neighborhood will force the same thing for all $v(q)$ where $q$ is
in some neighborhood of $p$. In other words, both $P_*$ and $P_x$
are open. Since paths in $\mathbb{R}^2$ are connected, one of
those is empty and the other is $P$. Since $v_l \in P_x, \; P_x =
P$. Recall that $P$ ends at some $(x,y)$ with $x$ in a forbidden
zone. This contradicts the choice of $K$, and so completes the
proof.
\end {proof}

\section {RPT does not imply BD-N}
At some point in this paper, it should be stated what the Riemann
Permutation Theorem actually is.

\begin {theorem} (Riemann) If every permutation of a series of
real numbers converges, then the series converges absolutely.
\end {theorem}

For a constructive analysis of the issues involved with
convergence of series, see \cite {BeB3, BBDS}. These include a
proof that BD-N implies RPT, as well as that absolute convergence
follows from merely having a bound on the partial sums of the
absolute values, which we use implicitly below.

In \cite {RSL} a topological model falsifying BD-N is presented,
as well as a proof that, in that model, the anti-Specker spaces
are closed under products. It is predictable that the proofs that
other properties slightly weaker than BD-N hold in the same model
would be very similar, and also true. To make the current paper
somewhat self-contained we will describe the underlying
topological space again; the argument afterwards that RPT holds
should seem familiar to anyone familiar with the anti-Specker
closure section from \cite {RSL}.

The points in $T$ be the functions $f$ from $\omega$ to $\omega$
with finite range, that is, enumerations of finite sets. A basic
open set $p$ is (either $\emptyset$ or) given by an unbounded
sequence $g_p$ of integers, with a designated integer ${\rm
stem}(p)$, beyond which $g_p$ is non-decreasing. $f \in p$ if
$f(n) = g_p(n)$ for $n < {\rm stem}(p)$ and $f(n) \leq g_p(n)$
otherwise. Notice that $p \cap q$ is either empty (if $g_p$ and
$g_q$ through their
\begin{figure} [h]
\begin{center}
\usetikzlibrary{patterns,arrows}

\begin{tikzpicture}
\tikzset{
>=stealth', axis/.style={<->}, connection/.style={thick, dotted},
}
    \coordinate (y) at (0,4.5); \coordinate (x) at (9,0);
    \draw[<->] (y) node[above] {$\omega$} -- (0,0) --  (x) node[right] {$\omega$}; \path;
\foreach \point in
{(0.5,1),(1,2),(1.5,2),(2,2.5),(2.5,1.5),(3,2),(3.5,2.5)}
    {\filldraw [black]  \point circle (2pt);}
\node at (2,2.8) [right] {$g_{p}$};
\draw[connection]  (3.5,4)--(3.5,0) node[below]
{$\mathit{\mathrm{stem}(p)}$};
\foreach \x in {4,4.5,5,5.5,6}
    {
    \draw[pattern=north east lines,pattern color=black!10] (\x-0.25,0) rectangle (\x+0.25,3);
    \filldraw [black]  (\x,3) circle (2pt);
        }
\foreach \x in {6.5,7,7.5,8}
    {
    \draw[pattern=north east lines,pattern color=black!10] (\x-0.25,0) rectangle (\x+0.25,3.5);
    \filldraw [black]  (\x,3.5) circle (2pt);
    }
\node at (8.5,2) [right] {\dots};
\end{tikzpicture}
\end{center}
\caption{An artist's impression of an open set $p$.}
\end{figure}
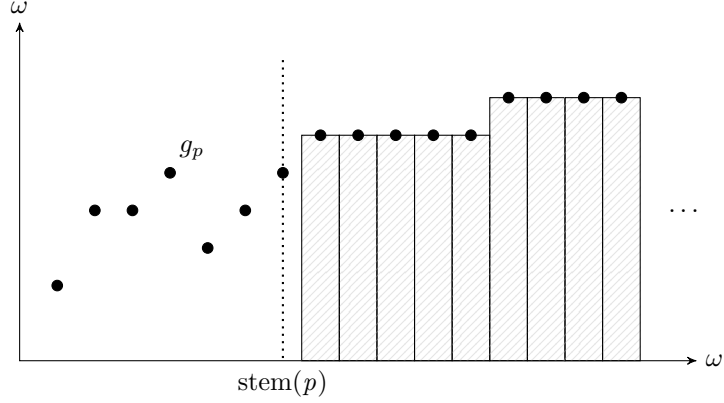
stems are incompatible) or is given by taking the larger of the
two stems, the function up to that stem from the condition with
the larger stem, and the pointwise minimum beyond that. Hence
these open sets do form a basis. It is sometimes easier to assume
that $g_p({\rm stem} (p)) \geq \max\{g_p(i) \mid i < {\rm
stem}(p)\}$. The intuition is that once a certain value has been
achieved there's nothing to be gained anymore by trying to
restrict future terms from being that big. It is not hard to see
that that additional restriction does not change the topology. So
whenever more convenient, a basic open set can be taken to be of
this more restrictive form.

\begin {theorem} $T \Vdash$ RPT.
\end {theorem}

\begin {proof}
Suppose $f \in p \Vdash ``$For every permutation $\sigma$ the
series $(a_{\sigma(n)})$ converges." It suffices to find a
neighborhood of $f$ forcing an upper bound for $\Sigma \left| a_n
\right|$. We assume as usual that $p$ is basic open and that
$g_p$(stem($p)) \geq \sup(\mathrm{rng}(f))$. So it suffices to
extend $p$ to $r$ forcing a bound for $\Sigma \left| a_n \right|$,
without altering the stem or the value $g_p$(stem($p))$ (i.e.
stem($p$) = stem($r$) and $g_p = g_r$ at their common stem), as
$f$ will then be in $r$. We can also assume that each $a_n$ is
rational, as $a_n$ could be replaced by a rational number (using
Countable Choice, which holds in this model \cite {RSL})
sufficiently close that convergence will not be affected.

\begin{definition} A finite sequence of integers $\sigma$ of length at
least ${\rm stem}(p)$ is {\bf compatible} with $p$ if for all $i <
{\rm stem}(p) \; \sigma(i) = g_p(i)$ and for all $i$ with ${\rm
stem}(p) \leq i < length(\sigma) \; \sigma(i) \leq g_p(i)$. For
$\sigma$ compatible with $p, \; p \upharpoonright \sigma$ is the
open set $q \subseteq p$ such that ${\rm stem}(q) =
length(\sigma),$ for $i < {\rm stem}(q) \; g_q(i) = \sigma(i),$
and otherwise $g_q(i) = g_p(i)$.
\end {definition}

The following lemma is analogous in statement and proof to lemma
3.3 from \cite {RSL}, the proof of which was an extrapolation of
some lemmas from an earlier section, which themselves were just
extensions of the basic lemma about this model. All of which is
meant to explain why the proof will not be repeated here.

\begin {lemma}
There is an open set $q \subseteq p$, with ${\rm stem}(q) = {\rm
stem}(p)$ and $g_q({\rm stem}(q)) = g_p({\rm stem}(p)),$ which
determines the values of $a_n$ in the following sense: for every
$n \in \mathbb{N}$ there is a length $i_n$ (increasing as a
function of $n$) such that, for all $\sigma$ of length $i_n$
compatible with $q$, $q \upharpoonright \sigma$ forces a
(rational) value for $a_n$, say $r_\sigma$.
\end {lemma}

Let $q$ be as in the lemma. The members of $q$ naturally form a
tree $Tr_q$: the nodes are those finite sequences compatible with
$q$, and the members of $q$ are those paths through the tree with
bounded range. At height $j \geq {\rm stem}(q)$ of $Tr_q$, the
amount of branching is $g_q(j)+1$. The nodes at height $i_n$
determine the value of $a_n$. We will have use for subsets of $q$
the members of which have ranges that are uniformly bounded. (Such
subsets are, of course, not open.) These subsets can be given as
the set of paths through a subtree $Tr$ of $Tr_q$ with a fixed
bound on the ranges of its nodes, as follows.

\begin {definition} A tree $Tr \subseteq Tr_q$ is {\bf bounded} if
there is a $J$ such that for all $\sigma \in Tr$ and $j <
length(\sigma) \; \sigma(j) < J.$
\end {definition}

The following is the analogue of \cite {RSL}'s lemma 3.5.

\begin {lemma} Let $Tr \subseteq Tr_q$ be bounded. Then there is a
bound $B$ in the sense that, for all $\sigma \in Tr$ of length
some $i_n$, $\Sigma_{m \leq n} \left| r_{\sigma \upharpoonright
i_m} \right| \leq B$.
\end {lemma}

\begin {proof}
Say that $\tau \in Tr$ is good if the conclusion of the lemma is
satisfied for $Tr \upharpoonright \tau$ (i.e. those nodes in $Tr$
extending $\tau$). Notice that if every immediate extension of
$\tau$ is good then $\tau$ itself is good, by taking the maximum
of the bounds witnessing the goodness of $\tau$'s extensions. So
if the root of $Tr$ is bad (i.e. not good) then there is a path
$f$ through $Tr$ consisting of all bad nodes. Because $Tr$ is
bounded, $f \in T$.

Define the permutation $\sigma$ as follows. At stage $n$, we have
inductively a permutation $\sigma_n$ of a natural number $l_n$
(with $l_0 = 0$). By the choice of $f$, there is an extension
$\tau$, with length some $i_k$, of $f \upharpoonright i_{l_n}$,
such that $\Sigma_{l_n < m \leq k} \left| r_{\tau \upharpoonright
i_m} \right|$ is at least 2. That can be only if either the sum of
the positive values of $r_{\tau \upharpoonright i_m}$ is at least
1, or that of the negative such values is at least 1. Without loss
of generality, suppose the former. Let $\sigma_{n+1}$ extend
$\sigma_n$ by first listing all of the $m$'s such that $r_{\tau
\upharpoonright i_m}$ is positive, and then listing all of the
other $m$'s which are at most $k$. As this $\sigma_{n+1}$ is a
permutation of $\{0, 1, ... , k\}$, the inductive construction can
continue. Letting $\sigma$ be $\bigcup_n \; \sigma_n$, no
neighborhood of $f$ can force that $\Sigma a_{\sigma(n)}$
converges.

\end {proof}

To complete the proof, let $Tr_1$ be the subtree of $Tr_q$ of all
nodes with values less than or equal to $I := g_q$(stem$(q))$.
Apply the lemma with $Tr_1$ for $Tr$ to get the least such upper
bound $B_1$. Now let $Tr_j \subseteq Tr_q$ extend $Tr_1$ by
allowing nodes that may take on the value $I+1$ at positions
beyond $j$. The lemma applied to $Tr_j$ produces a least upper
bound $B_j$. Notice that for $j<k$ we have $Tr_j \supseteq Tr_k
\supseteq Tr_1$, and so $B_j \geq B_k \geq B_1$. Let $B_\infty$ be
$\lim_j B_j$.

We claim that $B_\infty = B_1$. If not, then $\epsilon := B_\infty
- B_1 > 0$, and for all $j$ there are nodes $\tau \in Tr_1$ and
$\rho \in Tr_j$ extending $\tau$ such that, summing over the
$i_m$'s between the lengths of $\tau$ and $\rho$, $\Sigma \left|
r_{\rho \upharpoonright i_m} \right|$ is as close to $\epsilon$ as
you want. By a construction as in the previous lemma, a
permutation $\sigma$ could then be built with no condition forcing
$\Sigma_n a_{\sigma(n)}$ to converge. Hence $B_\infty = B_1$.
Choose $j$ so that $B_j$ is within 1/2 of $B_1$. Let $Tr_2$ be
$Tr_j$ and $B_2$ be $B_j$.

Continuing inductively, given $Tr_s$, build $Tr_{s+1}$ which
allows nodes to take on the value $I+s$ past a certain point and
has a lemma-induced least upper bound $B_{s+1}$ no greater than
2$^{-s}$ more than $B_s$. Let $Tr_\infty$ be $\bigcup_s Tr_s$.
$Tr_\infty$ induces an open set which forces $\Sigma_n \left| a_n
\right|$ to be bounded by $B_1 + 1.$ \end {proof}

\section {RPT may fail}
We now present a topological model in which RPT is false. First we
define the underlying space $T$, which not surprisingly will
involve reference to permutations. By way of notation, for
$\sigma$ a permutation of $\omega$, we think of $\sigma(n)$ as the
integer in the $n^{th}$ slot. So applying $\sigma$ to $0, 1, 2,
...$ would produce $\sigma(0), \sigma(1), \sigma(2), ...$.
Applying $\sigma$ to $(a_n)$ produces the sequence $a_{\sigma(0)},
a_{\sigma(1)}, ...$, for which we use the notation
$(a_{\sigma(n)})$.

\begin {definition} Let $T$ be the set of sequences $(a_n)$ which
are eventually 0 and which sum to 0.

An open set $O$ is given in part by a finite sequence $I_n (n<N)$
of intervals from $\mathbb{R}$, thought of as approximations to an
initial segment of $(a_n)$; that is, in order for $(a_n)$ to be in
$O$, it is necessary that $a_n \in I_n.$ Also, finitely many
permutation $\sigma$ are given. Each such $\sigma$ is associated
with finitely many pairs $\epsilon, M$, with $\epsilon > 0$ and $M
\in \mathbb{N}$. For $(a_n)$ to be in $O$, it must also be the
case that the partial sums $\Sigma_{n=0}^m a_{\sigma(n)}$ ($m
> M$) are less than $\epsilon$ in absolute value. In words, after
permuting $(a_n)$ by $\sigma$, the series must have converged to
within $\epsilon$ by $M$.
\end {definition}

\begin {theorem} $T \Vdash \neg$RPT.
\end {theorem}

\begin {proof}
The generic $G$ induces the generic sequence of reals $(g_n)$,
with $O \Vdash ``g_n \in I_n."$ Also, $T \Vdash ``(g_n)$ is
total," since, for every $(a_n) \in T$ and $k$, the open set
determined by any $(I_0, ... , I_k),$ with $a_n \in I_n$, and no
$\sigma$'s, forces $``g_k$ is defined." The generic sequence
$(g_n)$ will witness the failure of RPT.

First, we want to see that for every ground model permutation
$\sigma, \; \Sigma g_{\sigma(n)}$ converges. Notice that for every
$(a_n) \in T$ and $\epsilon > 0$ there an $M$ such that the open
set determined by associating $\epsilon$ and $M$ to $\sigma$
contains $(a_n)$, for the simple reasons that $\Sigma a_n$
converges to 0 and that $(a_n)$ is eventually 0: just choose $M$
so large so that all non-0 entries of $(a_n)$ have already
occurred in $a_{\sigma(n)}$ by the $M^{th}$ entry there. It
follows immediately that $T \Vdash ``\Sigma g_{{\hat \sigma}(n)} =
0."$

As for arbitrary permutations, suppose $O \Vdash ``\sigma$ is a
permutation." We claim that no extensions of $O$ can force
different values for any $\sigma(n)$; if that is so, then $O$
itself forces all of the values of $\sigma(n)$. To see the claim,
let $(a_n)$ and $(b_n)$ be two members of $O$. Consider the
continuous family of sequences $(c_n)^t := t(b_n) + (1-t)(a_n), \;
0 \leq t \leq 1.$ Notice that $(c_n)^0 = (a_n)$, $(c_n)^1 =
(b_n)$, and, for all $t$, $(c_n)^t \in O$, since the constraints
imposed by $O$ are linear. For any value of $t_0$ of $t$, some
neighborhood of $(c_n)^{t_0}$ forces a value for $\sigma(n)$. Any
such neighborhood forces the same value for all $(c_n)^t$ for $t$
in a neighborhood of $t_0$; that is, those $t$'s that force any
fixed value for $\sigma(n)$ form an open set. Since [0,1] is
connected, all $(c_n)^t$'s must have neighborhoods forcing the
same value for $\sigma(n)$. Hence the values of $\sigma(n)$ are
all determined by $O$. As the forcing relation is definable in the
ground model, these values form a ground model permutation. Since
all permutations are equal locally to ground model permutations,
by the previous paragraph, $\Sigma g_{\sigma(n)}$ converges for
all $\sigma$.

It remains only to show that $T \Vdash ``(g_n)$ diverges
absolutely." Consider any $(a_n) \in O$. There is a $K$ such that,
for any $\sigma$ which is a part of $O$'s definition, the partial
sums beyond $K$ of the permuted sequence are 0: for $k>K, \;
\Sigma_{n=0}^k a_{\sigma(n)} = 0.$ (It suffices to take $K = \max
\{\sigma^{-1}(n) | a_n \not = 0$ and $\sigma$ is constrained by
$O$ \}.) Choose some $i > \sigma"(K)$ (the image of $K$ under
$\sigma$), and change $a_i$ to be $\delta$, where $\delta$ is less
than all of the $\epsilon$-constraints imposed by $O$. Iterate to
find another safe spot $j$, and change $a_j$ to be $- \delta$.
This can be iterated to get the sum of the absolute values to be
as big as you want. Hence $O$ does not force any bound on the sum
of the absolute values.

\end {proof}

\section {That all partially Cauchy sequences are Cauchy does not imply BD-N}
Following a definition of Fred Richman (private notes), we say
that a sequence of reals $x_n$ is {\em partially Cauchy} if, for
all increasing $h$, $\lim_n \mathrm{diam}(x_n, x_{n+1}, ... ,
x_{h(n)}) = 0.$ (The diameter of a set in a metric space is the
supremum of the distances between members of the set, taken two at
a time, if this supremum exists. If the set is finite, as it is
here, the supremum does exist.) Richman showed, among other
things, that, under BD-N, every partially Cauchy sequence is
Cauchy. In this section we show that BD-N is not necessary for
this, in that the latter result does not imply BD-N. In the next
section, we show that the result in question is not provable in
basic set theory alone.

Let $T$ be the space from \cite {RSL}, reviewed in section 3
above, the model over which falsifies BD-N. As in the other cases,
we have:

\begin {theorem} $T \Vdash$ ``Every partially Cauchy sequence is
Cauchy."
\end {theorem}

\begin {proof}
Suppose $p \Vdash ``(x_n)$ is partially Cauchy." In a personal
communication, Fred Richman studied several notions of Cauchyness,
all akin to partiality, and showed essentially that any sequence
which is Cauchy in any sense (partially, weakly, almost) is the
sum of a Cauchy sequence (in as strong a sense as you like) and a
rational sequence which is Cauchy in the same sense as the
starting sequence. His proof uses Countable Choice, which is no
problem here, as $T \Vdash$ Dependent Choice (see \cite {RSL}),
which implies Countable Choice, and is otherwise straightforward.
The upshot of this is that we can assume that each $x_n$ is
rational.

For every $f \in p$ we must find a neighborhood $q$ of $f$ forcing
$(x_n)$ to be Cauchy. So let $T \Vdash \epsilon > 0.$ It suffices
to assume $\epsilon$ is rational, so we do not have to deal with
conditions forcing $\epsilon$ to have an approximate value. We
assume as usual that $p$ is basic open and that $g_p$(stem($p))
\geq \sup(\mathrm{rng}(f))$. So it suffices to extend $p$ to $r$
forcing an appropriate value $N$ for $\epsilon$, without altering
the stem or the value $g_p$(stem($p))$ (i.e. stem($p$) = stem($r$)
and $g_p({\rm stem} (p)) = g_r({\rm stem} (r))$), as $f$ will then
be in $r$.

As in section 3 above, we state without proof:

\begin {lemma}
There is an open set $q \subseteq p$, with ${\rm stem}(q) = {\rm
stem}(p)$ and $g_q({\rm stem}(q)) = g_p({\rm stem}(p)),$ which
determines the values of $x_n$ in the following sense: for every
$n \in \mathbb{N}$ there is a length $i_n$ (increasing as a
function of $n$) such that, for all $\sigma$ of length $i_n$
compatible with $q$, $q \upharpoonright \sigma$ forces a
(rational) value for $x_n$, say $r_\sigma$.
\end {lemma}

With terminology and notation as in section 3 above, we have the
following analogue of lemma 3.6.

\begin {lemma} Let $Tr \subseteq Tr_q$ be bounded, and $\delta >
0$ be rational. Then there is a natural number $k$ such that, for
all $m, n \geq k, m<n,$ and $\sigma_m \subseteq \sigma_n$ of
lengths $i_m$ and $i_n$ respectively, $\left| x_{\sigma_m} -
x_{\sigma_n} \right| < \delta.$
\end {lemma}

\begin {proof}
Say that $\tau \in Tr$ is good if the conclusion of the lemma is
satisfied for $Tr$ restricted to $\tau$. Notice that if every
immediate extension of $\tau$ is good then $\tau$ itself is good,
by taking the maximum of the $k$'s witnessing the goodness of
$\tau$'s extensions. So if the root of $Tr$ is bad (i.e. not good)
then there is a path $f$ through $Tr$ consisting of all bad nodes.
Because $Tr$ is bounded, $f$ is a member of the topological space
$T$.

Define the function $h$ as follows. Given $k$, let $m$ and $n$ be
as given by the badness of $f \upharpoonright k$ (i.e. there are
nodes $\sigma_m \subseteq \sigma_n$ in the tree beneath $f
\upharpoonright k$ with $\left| x_{\sigma_m} - x_{\sigma_n}
\right| \geq \delta).$ Let $h(k)$ be at least as big as that $n$
(and, for $k>0$, bigger than $h(k-1)$). Then $h$ witnesses that
$(x_n)$ is not partially Cauchy, as any neighborhood of $f$ must
contain all of $Tr$ restricted to some initial segment of $f$.
\end {proof}

To complete the proof, let $Tr_1$ be the subtree of $Tr_q$ of all
nodes with values less than or equal to $I := g_q$(stem$(q))$.
Apply the lemma with $Tr_1$ for $Tr$ and $\epsilon/2$ for
$\delta$. Let $k_1$ be the integer produced by the lemma. Let
$Tr_2 \subseteq Tr_q$ extend $Tr_1$ by allowing nodes that may
take on the value $I+1$ at positions beyond $i_{k_1}$. Again apply
the lemma, with $Tr_2$ for $Tr$ and $\epsilon/4$ for $\delta$, to
produce $k_2$, which can be taken to be larger than $k_1$. More
generally, at stage $s$, let $Tr_s \subseteq Tr_q$ extend
$Tr_{s-1}$ by allowing the value $I+s-1$ beyond $i_{k_{s-1}}$, and
let $k_s > k_{s-1}$ be the result of applying the lemma to $Tr_s$
and $\epsilon/2^s$.

To finish the definition of $r$, we must just give $g_r$, and show
that beyond $N := k_1$ $r$ forces the values of $(x_n)$ to be
within $\epsilon$ of one another. As motivation, consider $x_N$
itself, as compared with $x_m$ for some larger $m$ (larger than
$N$). If the value of $x_m$ is determined by some node in $Tr_1$,
we're golden -- even better than golden, $x_m$ being within
$\epsilon/2$ of $x_N$. But once we go into $Tr_2$, all bets are
off. Hence we want to restrict $Tr_r$ to equal $Tr_1$ at least for
nodes up to length $i_{k_2}$. If $m \leq k_2$, then $x_m$ is
determined by $Tr_1$, and we're done. For $m > k_2$, at least we
can bound $\left| x_N - x_{k_2} \right|$ by $\epsilon/2$, and work
on bounding $\left| x_{k_2} - x_m \right|$ by $\epsilon/4$, which
would suffice. While working on the latter, we can now afford to
be in the tree $Tr_2$. By continuing to expand the tree in which
we work in this fashion, we can guarantee that $g_r$ be unbounded,
while still remaining within $\epsilon$ of $x_N$.

So let $g_r$ between $i_{k_s}$ and $i_{k_{s+1}}$ have the value
$I+s-1$. This makes $g_r$ be unbounded, and forces the values of
$(x_n)$ beyond $N$ to be within $\epsilon$ of one another, by the
argument sketched in the previous paragraph.
\end {proof}

\section {Partially Cauchy sequences may not all be Cauchy}
As usual, in the coming topological model the generic will be a
partially Cauchy sequence which is not Cauchy. We start by
defining the underlying topological space.

\begin {definition}

Let $T$ be the space of all Cauchy sequences $(x_n)$. A basic open
set is given by finitely many pieces of information. One is a
finite sequence of intervals $I_n (n<k)$. A sequence $(x_n)$
satisfies the requirements $I_n (n<k)$ if for all $k<n \; x_n \in
I_n.$ In addition, to each of finitely many functions $h$ and
rational numbers $\epsilon > 0$ is associated a natural number
$n_{h, \epsilon}$. A sequence $(x_n)$ satisfies that requirement
if for all $n \geq n_{h, \epsilon} \; \mathrm{diam}(x_n, x_{n+1},
... , x_{h(n)}) < \epsilon$. The basic open sets as given are
closed under intersection, and so form a basis.

\end {definition}

\begin {theorem}
$T \Vdash$``Not every partially Cauchy sequence is Cauchy."
\end {theorem}

\begin {proof}
The generic induces a sequence $(g_n)$ of reals. For every $(x_n)
\in T$, ground model function $h$ from $\mathbb{N}$ to itself, and
rational $\epsilon > 0$, there is a neighborhood $O$ of $(x_n)$
assigning a value to $n_{h, \epsilon}$. Then $O \Vdash ``$ if $n
\geq n_{h, \epsilon}$ then $\mathrm{diam}(g_n, g_{n+1}, ... ,
g_{h(n)}) < \epsilon"$. Furthermore, in this model, all functions
from $\mathbb{N}$ to itself are ground model functions, by the
same argument as for permutations with respect to the RPT. Hence
the generic sequence is partially Cauchy.

To see that the generic sequence is not itself Cauchy, let $O$ be
an open set and $N$ an arbitrary natural number, which without
loss of generality is less than $k$, the length of $O$'s sequence
of intervals. We will show that $O$ does not force that beyond $N$
the values of the generic always stay within 1 of each other,
which suffices.

To simplify on notation (and thinking), we can strengthen (i.e.
shrink) $O$ by reducing to one $h$ (by taking the pointwise
maximum of the finitely many $h$'s) and one $\epsilon$ (by taking
the smallest). To be sure, this summary requirement does not
capture all of the actual requirements present before this
simplification, as there may have been demands made on intervals
starting at $n < n_{h, \epsilon}$. But those demands are only
finite in number, and can be satisfied by choosing $I_n (n < n_{h,
\epsilon})$ to be sufficiently small.

Pick a sequence of values $x_n$ of length $j := \max(n_{h,
\epsilon}, k)$ which is an initial segment of a member of $O$.
Extend that finite sequence to have a value just under $x_{j-1} +
\epsilon/2$ at the places $j$ through $h(j)$. Extend again to have
a value just under $x_{h(j)} + \epsilon/2$ at the places $h(j) +
1$ through $h(h(j) + 1)$. Extend again, by adding almost another
$\epsilon/2$ to the last value, from the next place, say $s$,
until $h(s)$. Continue this process for at least $2/\epsilon +
1$-many steps, at which point pick the Cauchy sequence which is
constant from that point on. The upshot is, beyond $N$ the
sequence has increased by more than 1.
\end {proof}

\end{document}